\documentclass[12pt,reqno,a4paper]{article}
 \usepackage{amsmath,amssymb,amsthm,amsfonts,epsfig,enumerate}
 \usepackage{mathtools}
 \usepackage{multicol}
 \usepackage[title]{appendix}
 \usepackage{color}
 \usepackage{enumerate}
 \usepackage{cite}
 \usepackage{hyperref}
 \hypersetup{linkcolor=blue, colorlinks=true ,citecolor = red}
\newcommand*\diff{\mathop{}\!\mathrm{d}}
\usepackage{amsmath,amssymb,amsthm,amsfonts,epsfig,enumerate}
\usepackage{mathtools}
\usepackage{multicol}
\usepackage{pst-plot,pst-node}
\usepackage{mathrsfs,amssymb}
\usepackage{graphicx}
\usepackage{color}
\newtheorem{definition}{Definition}[section]
\newtheorem{theorem}[definition]{Theorem}
\newtheorem{example}[definition]{Example}
\newtheorem{corollary}[definition]{Corollary}
\newtheorem{remark}[definition]{Remark}
\newtheorem{proposition}[definition]{Proposition}
\newtheorem{lemma}[definition]{Lemma}

\numberwithin{equation}{section}
\DeclarePairedDelimiter\abs{\lvert}{\rvert}%
\DeclarePairedDelimiter\norm{\lVert}{\rVert}%

\makeatletter
\let\oldabs\abs
\def\abs{\@ifstar{\oldabs}{\oldabs*}}
\let\oldnorm\norm
\def\norm{\@ifstar{\oldnorm}{\oldnorm*}}
\newcommand{\al} {\alpha}

\newcommand{\be} {\beta}

\newcommand{\De} {\Delta}

\newcommand{\om} {\omega}
\newcommand{\Om} {\Omega}

\newcommand{\Gr} {\nabla}
\newcommand{\no} {\nonumber}
\newcommand{\noi} {\noindent}

\newcommand{\ra} {\rightarrow}

\def\w{{\widetilde w}}

\def\dx{{\,\rm d}x}
\def\dt{{\rm d}t}
\def\dr{{\rm d}r}
\def\ds{{\rm d}s}

\def \tr {\textcolor{red}}

\def\sb2{{{\mathcal D}^{1,2}_0(B_1^c)}}

\def\w2r{{{ W}^{2,2}(\R^N)}}

\def\d2{{{\mathcal D}^{2,2}_0(\Om)}}

\def\Dr{{{\mathcal D}^{2,2}_0(\R^N)}}

\def\C{{\mathcal C}}
\def\D{{\mathcal D}}

\def\Me{{\mathcal M}}

\def\R{{\mathbb R}}
\def\N{{\mathbb N}}
\def\F{{\mathcal F}}
\def\({{\Big(}}
\def\){{\Big)}}
\def\ho{{H_0^1(\Om)}}

\def\ws2{{\F_{\frac{N}{2}}}}
\def\L2{{ L^{1,\;\infty}(\log L)^2}}

\def\l2{\mathcal M\log L}
\def\cc{{\C_c^\infty}}

\def\c1Loc{{\C_{loc}^1}}

\title{On the Generalized Hardy-Rellich Inequalities}
\author{T.V. Anoop \thanks{corresponding author}\,,
{Ujjal Das},
Abhishek Sarkar \footnote{The author was supported by the project LO1506 of the Czech Ministry of Education, Youth and Sports.}}

\date{}
\begin{document}
 \maketitle
 \begin{abstract}
 In this article, we look for  the weight functions (say $g$) that admits the following generalized Hardy-Rellich type inequality:
 \begin{equation*}
 \int_{\Omega} g(x) u^2 \dx \leq C \int_{\Omega} |\Delta u|^2 \dx, \ \forall u \in \mathcal{D}^{2,2}_0(\Omega),
 \end{equation*}
 for some constant $C>0$, where $\Omega$ is an open set in $\R^N$ with $N\ge 1$.  We find various  classes of such  weight functions, depending  on the dimension $N$ and the geometry of  $\Om.$ Firstly, we use  the Muckenhoupt condition for the one dimensional weighted Hardy inequalities and  a symmetrization  inequality to obtain  admissible weights in certain  Lorentz-Zygmund spaces. Secondly, using the fundamental theorem of integration  we obtain the weight functions  in certain weighted Lebesgue spaces. As a consequence of our results, we obtain simple proofs for the  embeddings of $\d2$ into certain Lorentz-Zygmund spaces proved  by Hansson and later by Brezis and Wainger. 
 \end{abstract}


 \section{Introduction and Main Results}
In this article, we  discuss the  generalized Hardy-Rellich inequalities. More precisely, we look for the weight functions $g$ that satisfy the following inequality:
\begin{align}\label{GHR}
\int_{\Om} g(x) u^2 \dx \le C \int_\Om |\De u|^2 \dx, \quad \forall u\in \d2,
\end{align}
where $\Om$ is an open set in $\R^N$ with $N\ge 1,$ $g\in L^1_{loc}(\Om)$ and $\d2$ is the completion of $\cc(\Om)$ with respect to 
 $\norm{\De u}_{L^2(\Omega)}$.  Depending on the dimension $N$ and on the geometry of $\Om,$ we find various classes of weight function that satisfies  \eqref{GHR}.

  The restriction on the dimension is mainly  due to the fact that the Beppo-Levi space $\d2$ may not be a function space for a general unbounded open set $\Om$. For example, when $1 \leq N\le 4,$ Hormander-Lions in \cite{Hormander} showed that $\Dr$ contains objects that do not belong to even in the space of distributions. However, when $\Om$ is an exterior domain we will see that (Remark \ref{embeddings})  $\d2$ is a well defined function space for any dimension $N.$ On the other hand, if $N\ge 5$ or $\Om$ is bounded, then  $\d2$ is always a function space and it is embedded in to  certain Lebesgue spaces. Thus depending on $N$, we will be
  considering various types of $\Om$ that ensures  the Beppo-Levi space $\d2$  is a function space:
  \begin{enumerate}[(i)]
  \item for $N\ge 5$: $\Om$ is an open set (bounded or unbounded),
  \item for $2\le N\le4$: $\Om$ is a bounded open set or an exterior domain.
  \end{enumerate}  
  Having made the assumptions on $N$ and $\Om$, we next look for conditions on $g$ so that \eqref{GHR} holds.

 First, recall the following classical Hardy-Sobolev inequality:
\begin{align}\label{Hardy}
 \int_\Om \frac{| u(x)|^2}{|x|^2}\dx \le \left(\frac{2}{N-2}\right)^2 \int_\Om |\Gr u|^2 \dx,\  \forall\, u\in \ho,
\end{align}
where $\Om$ is an open set in $\R^N (N \geq 3)$ containing the origin. Many  proofs for \eqref{Hardy}  are available in the literature. For an excellent  review of this topic we refer to the book \cite{Kufner}. Hardy-Sobolev inequality has been extended and generalized in several directions and for different function spaces. The improved Hardy-Sobolev inequalities are the ones that concerns with replacing the Hardy potential $\frac{1}{|x|^2}$ with $\frac{1}{|x|^2}$ + lower order radial weights,  see \cite{Brezis_Vazquez,Adimurthy_Mythily,Filippas} and the references therein.
On the other hand, many authors are also interested in generalized Hardy- Sobolev inequalities, i.e., more general weight functions in \eqref{Hardy} in place of  $\frac{1}{|x|^2}$. For example,  weights in certain Lebesgue spaces \cite{Manes-Micheletti, Allegretto},  weak Lebesgue spaces \cite{Visciglia} ($N\ge 3$) and  Lorentz-Zygmund spaces \cite{anoop} ($N=2$ and $\Om$ is a bounded). In this article, we study  the second order generalization of the Hardy-Sobolev inequality, namely \eqref{GHR}.
%
For brevity, we  make the following definition: 
 \begin{definition} \rm A function $g$ that satisfies \eqref{GHR}  is called an admissible weight. 
 \end{definition}
 
 Notice that the admissibility of $g^+$  ensures the admissibility of $g$,  henceforth in this article, we will be considering nonnegative  admissible weight functions. A nonnegative admissible function   necessarily belongs to  $L^1_{loc}(\Om)$.  The following second order generalization (for $N\ge 5$) of the  classical Hardy-Sobolev inequality is due to Rellich  \cite{Rellich}: 
  \begin{equation}\label{Rellich}
 \int_\Om \frac{| u(x)|^2}{|x|^4}\dx \le \frac{16}{N^2(N-4)^2} \int_\Om |\De u|^2 \dx, \ \forall\, u\in \w2r.  
  \end{equation} 
  Thus $\frac{1}{|x|^4}$ is an admissible weight. The authors used the spherical harmonics  in \cite{Rellich} to obtain the inequality \eqref{Rellich}, (see Section 7, Chapter 2, page 90-101). 
 Thereafter, many improved Rellich inequalities are proved in the literature, for example see  \cite{AdimurthiJFA, SantraCCM, GhoussoubM1, Tarsi, TertikasZ}. For further readings on the improved Hardy-Sobolev (first order) and Hardy-Rellich inequalities we refer to the monograph \cite{GhoussoubM} and the references therein.  The lack of P\'{o}lya-Szeg\"{o}  type inequality for the second order derivatives is one of the main difficulties in proving the Hardy-Rellich inequality. In general, the Schwarz symmetrization of an $\w2r$ function  do not admit the second order weak derivatives, even if they do, the second order derivatives may not satisfy the P\'{o}lya-Szeg\"{o} type inequality, see \cite{Milman, Cianchi2} for more discussion on this. 
   
 
  The embeddings of $\d2$ provide admissible weights in the dual of a space associated with the target space in the embedding.  Moreover, a finer embedding (a smaller target space) gives a larger class of admissible weights. For example, for  $N\ge 5,$  the embedding of $\Dr$  into the Lebesgue space $L^{2^{**}}(\R^N)$ ($2^{**}=\frac{2N}{N-4}$) ensures that  $L^{\frac{N}{4}}(\R^N)$ functions are admissible as obtained in \cite{Stavrakakis}.  A finer embedding of $\d2$ into the Lorentz space $L^{2^{**},2}(\Om)\subset L^{2^{**}}(\Om) $ is also available, see  \cite{Milman}. The embedding of $\d2$ into a smaller space $L^{2^{**},2}(\Om)$ provides a bigger class of admissible functions, namely the Lorentz space  $L^{\frac{N}{4},\infty}(\Omega)$. In this article, we present a proof for the admissibility of functions in $L^{\frac{N}{4},\infty}(\Omega)$ without using the above embedding and then obtain the embedding as a simple consequence of the admissibility.  The following theorem is  one of our main results:
\begin{theorem} \label{higherdimesionthm}
 Let $\Om$ be an open set  in $\R^N$ with $N\ge 5$ and  $g$ be a nonnegative  function.  
 \begin{enumerate}[(i)]
  \item (A sufficient condition)  If $g \in L^{\frac{N}{4},\infty}(\Om),$ then $g$ is admissible.
  \item (A necessary condition) In addition, let $\Om$ be a ball centered at the origin or entire $\R^N $ and $g$ be radial, radially decreasing. Then $g$ is admissible,  only if  $g$ belongs to  $L^{\frac{N}{4},\infty}(\Om).$
 \end{enumerate}
\end{theorem}
   Our proofs mainly rely on the Muckenhoupt necessary and sufficient conditions (Theorem 1 and Theorem 2 of \cite{Muckenhoupt}) for the one dimensional weighted Hardy inequalities and a pointwise inequality for the symmetrization that  obtained in \cite{Cianchi} (see (1.14)) using the rearrangement inequality for the convolution due to O'Neil (see \cite{Oneil}). We refer to \cite{EKP}, for  similar inequalities for the higher order derivatives.  Further, Theorem \ref{higherdimesionthm}  provides  a simple  proof for the embedding of $\d2$ into the Lorentz space $L^{2^{**},2}(\Om)$ (see Corollary \ref{BLtoLor}). 
 
  The space $L^{\frac{N}{4},\infty}(\Omega)$ does not include all the admissible weights. In the next theorem, we exhibit another class of admissible weights. The analogous result for the first order Hardy- Sobolev inequalities is obtained in \cite{Edelson}(see Lemma 1.1). 
  \begin{theorem}\label{Thm:Integral}
 Let $\Om\subset \R^N$ with $N\ge 5$ and let $g$ be a nonnegative function on $\Om$. If there exists a function $w\in L^1((0,\infty),r^3)$ such that $g(x)\le w(|x|)$ for all $x\in \Om$,
then  $g$ is admissible.    
 \end{theorem}
 The proof of the above theorem is based on the fundamental theorem of integral calculus. Further, we give  examples of  admissible weights  to  show that the  classes of  admissible weights given by Theorem~\ref{higherdimesionthm} and Theorem \ref{Thm:Integral}  are not contained in one another. 
 
 As we mentioned  before, when $N=4,$ the space $\d2$ may not be a function space for a general unbounded open set $\Om.$ However, for a bounded open set $\Om$, the Beppo-Levi space $\d2$ coincides with the usual Sobolev space  $H^2_0(\Om).$  Further,   $$H^2_0(\Om) \hookrightarrow L^A(\Om),$$
 where $L^A(\Om)$ is the Orlicz space generated by the N-function $A(t)=e^{t^2}.$  Using this  embedding, one can show that  all the nonnegative functions  in the Orlicz space $L\log L(\Om)$ are admissible. In this case, we use a point wise inequality for the symmetrization and the Muckenhoupt conditions  for the  one dimensional weighted Hardy inequalities to obtain a bigger class of admissible weights.
 For a measurable function $g,$ we denote its decreasing rearrangement by $g^*$ and  $g^{**}(t)=\frac{1}{t}\int_0^t g^*(s) ds.$ Now we define 
 \begin{align*}
  \l2(\Omega):= \left\{ g {\;\rm measurable\; }: \sup_{0<t<|\Om|}t \,\log\left( \frac{|\Om|}{t} \right) g^{**}(t) < \infty \right\}.
 \end{align*}
$\l2(\Omega)$ is a rearrangement invariant Banach function space with the norm 
\[ \norm{g}_{\l2(\Omega)}= \sup_{0<t<|\Om|}t \,\log\left( \frac{|\Om|}{t} \right) g^{**}(t), \]
for more on Banach function space see \cite{BennettR}. 
Now we state our next result.
\begin{theorem} \label{dim4thm}
 Let $\Om$ be a bounded open set  in $\R^4$ and  let $g$ be a nonnegative  function. 
 \begin{enumerate}
  \item (A sufficient condition) If $g\in \l2(\Omega)$, then $g$ is admissible.
  \item (A necessary condition) In addition, let $\Om$ be a ball centered at the origin and let $g$ be radial, radially decreasing. Then $g$ is admissible, only if  $g$ belongs to  $\l2(\Omega)$.
 \end{enumerate}
      \end{theorem}
 As a simple consequence  of the above theorem we have Corollary \ref{CriticalEmbd},  which gives the   embedding of $H^2_0(\Om)$ into a Lorentz-Zygmund spaces (finer than the embedding to Orlicz spaces) obtained independently by Brezis and Wainger \cite{Brezis_Wainger} and Hansson \cite{Hansson}. 

  Next we consider the exterior domains and annular regions  in $\R^N$ with $2\le N \le 4$. In this case we have the following results:  
 \begin{theorem}\label{Thm:exterior}
 Let $\Om=B_R\setminus \bar{B_1}  \subset \R^N$ with $1\le R \le \infty.$ Let $g$ be a nonnegative function and $w$ be another function such that $g(x)\le w(|x|)$ for all $x\in \Om$. If  
 $$w\in  \left \{ \begin{array}{ll}
 L^1((1,\infty),r^{N+1}
 ), & N= 3,4; R=\infty \\
 L^1((1,\infty),r^3\log r),  & N=2; R=\infty\\
 L^1(1,R), & 2\le N\le 4; R<\infty.
 \end{array} \right. $$ then  $g$ is admissible.    
 \end{theorem}

 This article is organized as follows. In Section~\ref{Prelim}, we  briefly discussions  the function spaces and other prerequisites which are essential for the development of this article. Section~\ref{MainThm} deals with the proof of  Theorem~\ref{higherdimesionthm}-- Theorem \ref{Thm:exterior}. 
 In appendix, we present some results on Lorentz-Zygmund spaces which we require in this article.

%
%
%
%

\section{Preliminaries}\label{Prelim}
In this section, we first  describe the symmetrization and some of its properties, then we briefly discuss about the rearrangement invariant function spaces which will appear in this article. In the end, we discuss the Muckenhoupt conditions for  the one dimensional weighted Hardy inequalities.  
\subsection{Symmetrization}
Let $\Omega\subset \R^N$ be a Lebesgue measurable set.  Let $\Me(\Om)$ be the set of all extended real valued Lebesgue measurable functions those are finite a.e. in $\Om.$  For $f\in \Me(\Om)$ and for $s>0$, we define
$E_f(s)=\{x: |f(x)|>s \}.$  Then the {\it distribution function} $\alpha_f$ of $f$ is defined as 
\begin{eqnarray*}
\alpha_f(s) &: =&
 \big\vert E_f(s) 
 \big\vert, \, \mbox{ for } s>0,
\end{eqnarray*}
where $|A|$ denotes the Lebesgue measure of a set $A\subset \R^N.$
 Now we define the {\it one dimensional decreasing rearrangement} $f^*$ of $f$ as below: 
 \begin{align*}
f^*(t):= \begin{cases*} \operatorname{ess}\ \sup f, \ \ t =0\\ \inf \{s>0 \, : \, \alpha_f(s) < t \}, \; t>0.   \end{cases*}                   
 \end{align*}
The map $f \mapsto f^*$ is not sub-additive. However, we obtain a sub-additive function from $f^*,$ namely the maximal function $f^{**}$ of $f^*$, defined by 
\begin{equation*}
f^{**}(t)=\frac{1}{t}\int_0^tf^*(\tau) d\tau, \quad t>0.
\end{equation*}
The sub-additivity of $f^{**}$ with respect to $f$ helps us to define norms in certain function spaces.

The { \it Schwarz symmetrization } of $f$ is defined by 
\begin{equation*}
  f^\star(x)=f^*(\omega_N|x|^N)  \label{relation}, \quad \forall\, x\in \Omega^\star,
\end{equation*}
where $\omega_N$ is the measure of the unit ball in $\R^N$ and $\Omega^\star$ is the open ball centered at the origin with same measure as $\Omega.$

Next we state an important inequality concerning the Schwarz symmetrization, see Theorem 3.2.10 of \cite{EdEv}.
\begin{proposition}[Hardy-Littlewood inequality] 
Let $\Omega \subset \R^N$ with $N\ge 1$ and $f$, $g$ be nonnegative measurable functions. Then
\begin{align} \label{HardyLittle}
 \int_{\Omega} f(x)g(x) \dx \leq \int_{\Omega^\star}f^\star(x)g^\star(x) \dx = \int_0^{|\Omega|}f^*(t) g^*(t)\dt.
\end{align}
  \end{proposition}

\subsection{Lorentz spaces}
The Lorentz spaces are refinement of usual Lebesgue spaces introduced by Lorentz himself in\cite{Lorentz}. For more details on Lorentz spaces and related results, we refer to the books \cite{Adams,EdEv,Loukas} and the article \cite{Hunt}. 

Let $\Om$ be an open set in $\R^N.$ Given a function $f\in\mathcal{M}(\Omega)$ and $(p,q) \in [1,\infty)\times[1,\infty]$ 
we consider the following quantity:
\begin{align*} 
 |f|_{(p,q)} := \norm{t^{\frac{1}{p}-\frac{1}{q}} f^{*} (t)}_{{L^q((0,\infty))}}
=\left\{\begin{array}{ll}
         \left(\displaystyle\int_0^\infty \left[t^{\frac{1}{p}-\frac{1}{q}} {f^{*}(t)}\right]^q \dt \right)^{\frac{1}{q}};\; 1\leq q < \infty, \vspace{4mm}\\ 
         \displaystyle\sup_{t>0}t^{\frac{1}{p}}f^{*}(t);\; q=\infty.
        \end{array} 
\right.
\end{align*}
The Lorentz space $L^{p,q}(\Om)$ is defined as
\[ L^{p,q}(\Om) := \left \{ f\in \mathcal{M}(\Om): \,   |f|_{(p,q)}<\infty \right \}.\]
$ |f|_{(p,q)}$ is  a complete quasi norm on $L^{p,q}(\Om).$ 
For $(p,q) \in (1,\infty)\times[1,\infty]$, let  
 $$ \norm{f}_{(p,q)}:= \norm{t^{\frac{1}{p}-\frac{1}{q}} f^{**} (t)}_{{L^q((0,\infty))}}.$$
Then $\norm{f}_{(p,q)}$ is a norm on $L^{p,q}(\Om)$ and it is equivalent to the quasinorm $|f|_{(p,q)}$ (see Lemma 3.4.6 of \cite{EdEv}). 
 For the computational simplicity,  we use  $ |f|_{(p,q)}$ instead of  $\norm{f}_{(p,q)}$.
 Note that $L^{p,p}(\Omega)=L^p(\Omega)$  for $p\in (1,\infty)$ and  $ L^{p,\infty}(\Om)$ coincides with the weak-$L^p$ space (Marcinkiewicz space)
 $:=\left \{ f \in \mathcal{M}(\Om):\sup_{s>0}s(\alpha_f(s))^{\frac{1}{p}}<\infty \right \}.$
 
 \subsection{Lorentz-Zygmund Space}
Now we briefly sketch an overview of Lorentz-Zygmund spaces. For more details on Lorentz-Zygmund Spaces we refer \cite{Bennet, BennettR, Edmund}. For a bounded open set $\Omega \subset \R^N$, $1\le p,q\le \infty$, and $\al\in \R$ we define the following quasinorms:
$$|f|_{(p,q,\al)}=\norm{\bigg(\log\left(\frac{e|\Omega|}{t}\right)\bigg)^\al t^{\frac{1}{p}-\frac{1}{q}} f^{*}(t)}_{L^q((0,|\Om|))}.$$
Then the Lorentz-Zygmund space $L^{p,q}(\log L)^\al(\Om)$  is defined as  
$$L^{p,q}(\log L)^\al(\Om):=\left \{f \in \mathcal{M}(\Om): |f|_{(p,q,\al)}< \infty\right \}.$$ For $p,q$ and $\al$ as before, let
$$\norm{f}_{(p,q,\al)}=
          \norm{\bigg(\log\left(\frac{e|\Omega|}{t}\right)\bigg)^\al t^{\frac{1}{p}-\frac{1}{q}} f^{**}(t)}_{L^q((0,|\Om|))}.$$   
For  $p>1,$ $\norm{f}_{(p,q,\al)}$ is a norm on $L^{p,q}(\log L)^\al(\Om)$  and it is equivalent to the quasinorm  $|f|_{(p,q,\al)}$ ( see Corollary 8.2 of \cite{Bennet}). In appendix (Proposition \ref{equivnorm}), we provide a proof  for the equivalence in  the case $p=\infty, q=2$ and $\al=-1.$

Note that, $L^{p,q}(\log L)^0(\Om)$ coincides with the Lorentz space $L^{p,q}(\Om)$. In appendix, we show that as a vector space $\L2(\Omega)$ and $\l2(\Omega)$ are same (Proposition \ref{equiv}). However, the quasinorm $|f|_{(1,\infty,2)}$ and the norm $\norm{f}_{\l2}$ are not equivalent.

\subsection{Muckenhoupt Condition}
The following necessary and sufficient conditions (see Theorem 1 and Theorem 2 of \cite{Muckenhoupt}) for the one dimensional weighted Hardy inequalities play an important role in our results:

\begin{theorem}[Muckenhoupt condition]\label{Muckenhoupt_condition}
Let $u,v$ be nonnegative measurable functions such that $v>0$. Then for any $a\in (0,\infty],$
\begin{enumerate}[(i)] 
 \item  the inequality
 \begin{align}
 \int_0^a  \left  | \int_0^s f(t) \dt \right |^2 u(s) \ds \leq C \int_0^a |f(s)|^2 v(s) \ds,  \label{Muck1}
\end{align}
 holds for all measurable function $f$ on $(0,a)$ if and only if 
\begin{align} \label{Muckconst1}
A_1:=\sup_{0<t<a} \left(\int_t^a u(s) \ds \right)\left ( \int_0^t v(s)^{{-1}} \ds \right)< \infty.
\end{align}
\item the dual inequality

\begin{align}
 \int_0^a  \left  | \int_s^a f(t) \dt \right |^2 u(s) \ds \leq C \int_0^a |f(s)|^2 v(s) \ds,     \label{Muck2}
\end{align}
holds for all measurable function $f$ on $(0,a)$ if and only if 
\begin{align} \label{Muckconst2}
A_2:=\sup_{0<t<a} \left(\int_0^t u(s) \ds \right)\left ( \int_t^a v(s)^{{-1}} \ds \right)< \infty.
\end{align}

\end{enumerate}

\end{theorem}
 
\begin{remark}\rm
	Let $C^1_b$ and $C_b^2$ denote the best constants in \eqref{Muck1} and \eqref{Muck2} respectively. Then we have the following inequality (see \cite{Kufner})
	\begin{equation}
	{A_i} \leq C^i_b \leq 2{A_i}, \text{ for } i=1,2,    \label{Muckbest}
	\end{equation} where $A_i$'s are defined in \eqref{Muckconst1} and \eqref{Muckconst2}. 
	\end{remark}

  \section{Proof of main theorems} \label{MainThm}
In this section we prove our main theorems. First, we state an inequality (1.14 of \cite{Cianchi}) that plays the role of  P\'{o}lya-Szeg\"{o} inequality for the second order derivatives. This inequality is obtained using the rearrangement inequality for the convolution due to O'Neil \cite{Oneil}. 
 \begin{lemma}\label{Lem:Cianchi}
 For $u\in \cc(\R^N)$ with $N\ge 3$, let $u^*$ be the decreasing rearrangement of $u.$ Then the following inequality holds:
\begin{eqnarray} \label{Cianchiineq}
 u^*(s) \leq \frac{1}{2(N-2)\om_N^{{\frac{2}{N}}}} \left (s^{-1+\frac{2}{N}} \int_0^s |\Delta u|^*(t)\dt + \int_s^{\infty} |\Delta u|^*(t)t^{-1+\frac{2}{N}}\dt \right),\,\forall \,s>0.
\end{eqnarray}
\end{lemma}
The next lemma is a consequence of  the Muckenhoupt condition:
 \begin{lemma}\label{Lem:higher}
	For $N\ge 4,$ let $\Om$ be an open set in $\R^N$. In addition, let $\Om$ be  bounded  when $N=4.$  Then for $$g \in X:= \left\{\begin{array}{ll}
	                   L^{\frac{N}{4},\infty}(\Omega), N\ge 5,\\
	                  \l2(\Omega), N=4.
	                  \end{array}\right.,$$ there exists a constant $C=C(N)>0$ such that  the following two inequalities hold: 
	\begin{equation}
\int_0^{|\Om|} g^*(s)s^{-2+\frac{4}{N}} \left(\int_0^s f(t)\dt\right)^2 \ds \leq C \norm{g}_X  \int_0^{|\Om|} f(s)^2 \ds ,  \label{M.eq1}
\end{equation}
\begin{equation}
\int_0^{|\Om|} g^*(s)  \left( \int_s^{|\Om|}f(t) t^{-1+\frac{2}{N}} \dt\right)^2 \ds \leq C \norm{g}_X \int_0^{|\Om|} f(s)^2 \ds, \label{M.eq2}
\end{equation}
 for  any measurable function $f$ on $(0,|\Om|)$.
\end{lemma}
\begin{proof}
	For proving \eqref{M.eq1},  we set $a=|\Om|, u(s)=g^*(s) s^{-2+\frac{4}{N}}$ and $v(s)=1$ in \eqref{Muck1}.  Thus  $ \int_0^t v(s)^{{-1}} \ds= \int_0^t \ds =t$ and 
 $$
\int_t^a u(s) \ds  = \int_t^{|\Om|} g^*(s) s^{-2+\frac{4}{N}} \le   g^*(t) \int_t^{|\Om|}  s^{-2+\frac{4}{N}}\ds
 = \left\{\begin{array}{ll} 
                                                                                                                          \frac{N}{ N-4} t^{\frac{4}{N}-1}g^{*}(t), N\ge 5, \\ 
                                                                                                                          \log(\frac{|\Om|}{t})g^{*}(t), N=4.
                                                                                                                          \end{array}\right.$$
 Therefore,  $$A_1= \sup_{0<t<a} \left(\int_t^a u(s) \ds \right)\left ( \int_0^t v(s)^{{-1}} \ds \right)\le C \norm{g}_X<\infty$$ and hence 
 \eqref{M.eq1} follows from part (i) of Theorem \ref{Muckenhoupt_condition}.\\
    \noi  To prove \eqref{M.eq2} we set $a=|\Om|,u(s)=g^*(s) $ and $v(s)=s^{2-\frac{4}{N}}$ in \eqref{Muck2}. Now
    $ \int_0^t u(s) \ds=\int_0^t g^*(s)\ds= tg^{**}(t)$ and 
         \begin{equation*}\label{eqn3}
  \int_t^a v(s)^{-1} \ds = \int_t^{|\Om|} s^{-2+\frac{4}{N}}\ds \le \left\{\begin{array}{ll} 
                                                                                                                          \frac{N}{ N-4} t^{\frac{4}{N}-1}, N\ge 5, \\ 
                                                                                                                          \log(\frac{|\Om|}{t}), N=4.
                                                                                                                          \end{array}\right.
  \end{equation*}
  Therefore,
     $$A_2=\sup_{0<t<a} \left(\int_0^t u(s) \ds \right)\left ( \int_t^a v(s)^{{-1}} \ds \right)\le C \norm{g}_X<\infty.$$
 Hence 
 \eqref{M.eq2} follows from part (ii) of Theorem \ref{Muckenhoupt_condition}.
	\end{proof}

\subsection{The higher dimension, \texorpdfstring{$N\geq 5$}{Lg}}

 In this subsection, we give  proofs of Theorem \ref{higherdimesionthm} and Theorem \ref{Thm:Integral}. 
\begin{proof}[\bf Proof of Theorem~\ref{higherdimesionthm}]
(i) {\bf A sufficient condition}. Let $u\in  \cc(\Om)$. Then by the Hardy-Littlewood inequality \eqref{HardyLittle} we have 
\begin{equation}\label{Hardy-Littlewood}
\int_\Om g(x) u(x)^2\dx \le \int_0^{|\Omega|} g^*(s) u^*(s)^2\ds.
\end{equation}  
Further \eqref{Lem:Cianchi} gives
\begin{align}\label{reardec}
 \int_0^{|\Omega|} g^*(s) u^*(s)^2 \ds\leq & \ 2\int_0^{|\Omega|} g^*(s)s^{-2 +\frac{4}{N}} \left(\int_0^s|\Delta u|^*(t)\dt \right)^2\ds   \notag \\ 
       \quad &+ 2\int_0^{|\Omega|} g^*(s) \left(\int_s^{\infty}|\Delta u|^*(t)t^{-1+\frac{2}{N}}\dt \right)^2\ds.  
 \end{align}
 Since $g\in  L^{\frac{N}{4},\infty}(\Omega)$, using  Lemma~\ref{Lem:higher} we can bound the  right hand side of the inequality  by $C \norm{g}_{(\frac{N}{4},\infty)} \int_0^{|\Om|}\left(|\Delta u|^*(t)\right)^2 \dt.$ As
 $\norm{|\Delta u|^*}_{L^2((0,|\Om|))}=\norm{\Delta u}_{L^2(\Om)}$,  \eqref{Hardy-Littlewood} and \eqref{reardec} yields
\begin{align}\label{sufficiency:eqn1}
\int_\Om g(x) u(x)^2\dx \leq & C \norm{g}_{(\frac{N}{4},\infty)} \int_\Om |\De u|^2 \dx, \; \forall u\in  \cc(\Om).
 \end{align}
Thus by  density of  $\cc(\Om)$ in $\d2$, the above inequality holds for all $u$ in $\d2$ and hence  $g$ is admissible.

\noi(ii) {\bf A necessary condition}.  Let $R\in (0,\infty]$ and let $\Om=B(0;R)\subset \R^N$ with $N\ge 5$. Let $ g:\Om \ra [0,\infty)$ be a radial and radially decreasing  admissible function. 
We will show that $g \in L^{\frac{N}{4},\infty}(\Omega)$. For each $r\in (0,R),$ consider the following function:
 \begin{align*}
           u_r(x)= \left\{
          \begin{array}{ll}
           (r-|x|)^2 & ; |x|\leq r, \\
                                                          0 &; \text{ otherwise.}
                                                          \end{array}\right .
 \end{align*}
By differentiating twice, we get
  \begin{align*}
           \De u_r(x)= \left\{
          \begin{array}{ll}
           2N - (2N- 2)\frac{r}{|x|} &; \ |x|< r, \\
                                                         0 &;\text{ otherwise.}
                                                          \end{array}\right .
 \end{align*}
 Now
    \begin{align}
    	\int_{\Om} |\De u_r|^2 \dx &= \int_{B_r}|\De u_r|^2 \dx = \int_{B_r}\left[2N - (2N- 2)\frac{r}{|x|}\right]^2 \dx\notag \\&\leq 2 \bigg[4N^2 \om_N r^N +(2N-2)^2r^2 \int_{B_r} {\frac{1}{|x|^2}}\dx\bigg] \notag\\&\leq C_1\bigg[r^N + r^2 \int_0^r s^{N-3}\ds\bigg] \leq C_2r^N, \label{necessary2}
    \end{align}  
 where $C_1, C_2$ are constants that depends only on $N$. Thus for each $r\in (0,R),$ $u_r\in \d2.$ Furthermore, by the admissibility of $g$, we have 
 \begin{equation}\label{necessary3}
  \int_\Om g(x)u_r^2 \dx \leq C \int_\Om |\De u_r|^2 \dx,\; \forall r\in(0,R).
 \end{equation}
 Since $g$ is  radial and radially decreasing, the left hand side of the above inequality  can be estimated as below:        
   \begin{align}
   	\int_{\Om} g(x) u_r^2 \dx &\geq \int_{B_{{\frac{r}{2}}}}g(|x|) u_r^2 \dx \geq  \left(r-{\frac{r}{2}}\right)^4\int_{B_{{\frac{r}{2}}}}g(|x|)\dx  \notag \\&= \left({\frac{r}{2}}\right)^4\int_{B_{{\frac{r}{2}}}}g^{\star}(x)\dx =\left({\frac{r}{2}}\right)^4 \int_0^{\om_N ({\frac{r}{2}})^N} g^* (s)\ds.   \label{necessary1}
  	\end{align}    
  From \eqref{necessary2}, \eqref{necessary3} and \eqref{necessary1}, we obtain
    $$\left({\frac{r}{2}}\right)^4 \int_0^{\om_N ({\frac{r}{2}})^N} g^*(s)\ds \leq C C_2 r^N.$$ 
  Now by setting $\om_N ({\frac{r}{2}})^N=t $ and since $0<r<R$ is arbitrary, we conclude that $$ \sup_{t \in (0,\frac{|\Om|}{2^N})} t^{{\frac{4}{N}}} g^{**}(t) \leq C_3.$$ As $ t^{{\frac{4}{N}}} g^{**}(t) $  is bounded on $ (\frac{|\Om|}{2^N}, |\Om|) $,
  $g$ must belong to $L^{\frac{N}{4},\infty}(\Omega)$.
          \end{proof}
      
\begin{remark}\label{best}{\rm 
	Let $C_R$ be the best constant in \eqref{GHR}. Then  from \eqref{Cianchiineq}, \eqref{Muckbest} and Lemma~\ref{Lem:higher} one can deduce that
\begin{equation*}
 C_R \leq \frac{N }{(N-4)(N-2)^2 \om_N^{\frac{4}{N}}}\norm{g}_{(\frac{N}{4},\infty)}.
\end{equation*}
}
\end{remark}
\begin{example}\label{Hardy-Rellich}
{\rm  For  $\al\in (0,N)$ and  $R\in (0,\infty]$ let  $g(x)= \displaystyle\frac{1}{|x|^\al}, x\in B(0;R).$ It is easy to calculate
$$g^{*}(t)= \left\{\begin{array}{ll}
\left(\frac{\om_N}{t}\right)^{\frac{\al}{N}} & 0<t<\om_N R^N,\\
                    0 & t\ge \om_N R^N                   .
\end{array} \right. \quad  \quad g^{**}(t)= \left\{\begin{array}{ll}
                \frac{N}{N-\al}\left(\frac{\om_N}{t}\right)^{\frac{\al}{N}}& 0<t<\om_N R^N,\\
                 0 & t\ge \om_N R^N \end{array} \right.$$
Therefore,  $$ g\in L^{\frac{N}{4},\infty}(B(0;R)) \text{ with } \left\{\begin{array}{ll}
 R<\infty  \text{ if and only if}  & \al\le 4 \\
R=\infty  \text{ if and only if}  & \al=4.
\end{array}\right.$$ }
\end{example}
\begin{remark}
{\rm 
From the above example, it is clear that $g(x)=\frac1{|x|^4}$ belongs to $L^{\frac{N}{4},\infty}(\R^N)$ and $\norm{g}_{(\frac{N}{4},\infty)}=   \frac{N \om_N^{\frac{4}{N}}}{N-4}.$  Thus the  Hardy-Rellich inequality  \eqref{Rellich}  follows  easily from  part (i) of Theorem \ref{higherdimesionthm}.  Further, the best constant  in \eqref{Rellich}  equals to  $\frac{16}{(N-4)^2 N^2}$ which is  bounded by the constant $\frac{N^2}{(N-4)^2(N-2)^2}$  given by Remark \ref{best}. 
}
\end{remark}
 As a consequence  of the sufficiency part of  Theorem \ref{higherdimesionthm}, we have  a simple   proof  for the following Lorentz-Sobolev embedding: 
\begin{corollary}\label{BLtoLor}
	Let $\Omega \subset \mathbb{R}^N$ is an open set and $N \geq 5$. Then we have the following embedding:$$D^{2,2}_0(\Om) \hookrightarrow L^{2^{**},2}(\Om), \text{ where } 2^{**}=\frac{2N}{N-4}.$$
\end{corollary}
\begin{proof} Without loss of generality we may assume $\Om = \R^N$ (for a  general domain $\Om$, the result will follow by considering the zero extension to $\R^N$). By \eqref{sufficiency:eqn1}, for each $ g\in L^{\frac{N}{4},\infty}(\R^N)$  we have $$ \int_0^{\infty} g^*(t) (u^*(t))^2 \dt \leq C \norm{g}_{(\frac{N}{4},\infty)} \int_{\R^N} |\De u|^2 \dx,\, \forall\, u\in \Dr.$$
	In particular,  if we choose $g(x)=\frac{1}{|x|^4},$ then $g^*(t)=\left(\frac{\om_N}{t}\right)^{\frac{4}{N}}$ and $\norm{g}_{(\frac{N}{4},\infty)}=   \frac{N \om_N^{\frac{4}{N}}}{N-4}.$  Now by substituting in the above inequality, we get 
	$$\int_0^{\infty} t^{-\frac{4}{N}} (u^*(t))^2 \dt \leq C_1  \int_{\R^N} |\De u|^2 \dx, \, \forall\, u\in \Dr,$$ where $C_1$ is a constant that depends only on $N.$ 	
	Since $\int_0^{\infty} t^{-\frac{4}{N}} (u^*(t))^2 \dt = |u|_{(2^{**},2)}^2$  is  equivalent to  $\norm{u}^2_{(2^{**},2)}$,we obtain the required embedding
	$$\norm{u}^2_{(2^{**},2)} \le C_2 \int_{\R^N} |\De u|^2 \dx,\  \forall u\in \Dr.$$ 
	 \end{proof}
The following  lemma is needed  for the  proofs  of   Theorems \ref{Thm:Integral} and \ref{Thm:exterior}. 
 \begin{lemma}\label{Polarinequality}
      	For $u \in C_c^\infty({\R^N})$, the following inequality holds: 
      	\begin{equation*}
      	\int_0^{\infty} \int_{\mathbb{S}^{N-1}} r^{N-1}\left|\frac{\partial^2 u}{\partial t^2}(r,w)\right|^2 \diff S_w \dr \leq \int_{\R^N} |\Delta u|^2 \dx 
      	\end{equation*}
      	\end{lemma}    
      \begin{proof}
      	      Observe that 
      	\[ \frac{\partial u}{\partial \eta} = \nabla u \cdot \eta \text{ and } \frac{\partial^2 u}{\partial \eta^2}= \nabla(\nabla u \cdot \eta)\cdot \eta= \sum_{i=1}^N \sum_{j=1}^N \frac{\partial^2 u}{\partial x_i \partial x_j} \eta_i \eta_j.\]
      	            	      Further, have the following  inequality for an $N \times N$ real matrix  $A=(a_{ij})$ and $x \in \mathbb{R}^N:$
      	      \begin{equation} \label{polartoeculid23}
      	      |\langle Ax,x \rangle |^2 \leq |\sum_{i=1}^N \sum_{j=1}^N a_{ij}x_i x_j |^2 \leq \bigg(\sum_{i=1}^N \sum_{j=1}^N a_{ij}^2\bigg) \bigg(\sum_{i=1}^N x_i^2 \bigg) \bigg(\sum_{j=1}^N x_j^2\bigg).
      	      \end{equation} 
      	      Now by writing $x=(r,\omega) \in (0,\infty) \times \mathbb{S}^{N-1}$ for  $x \in \R^N\setminus\{0\}$, and using \eqref{polartoeculid23}, we obtain 
      	\begin{align*}  
      		\int_0^{\infty} \int_{\mathbb{S}^{N-1}}\bigg| \frac{\partial^2 u }{\partial r^2}(r,\omega)\bigg|^2 r^{N-1} \diff S_{\omega} \dr &=  \int_{\R^N}\bigg |\frac{\partial^2 u}{\partial r^2}\bigg|^2 \dx= \int_{\R^N} \bigg(\sum_{i=1}^N \sum_{j=1}^N \frac{\partial^2 u}{\partial x_i \partial x_j} \frac{x_i}{|x|} \frac{x_j}{|x|} \bigg)^2 \dx \no\\
      		 &\leq \int_{\R^N} \sum_{i=1}^N \sum_{j=1}^N  \bigg( \frac{\partial^2u}{\partial x_i \partial x_j}\bigg)^2 \dx = \int_{\R^N} |\Delta u|^2 \dx,  
      	\end{align*} and this concludes the proof.
      	\end{proof}           

Next we  prove   Theorem~\ref{Thm:Integral}. 
\begin{proof}[{\bf  Proof of Theorem \ref{Thm:Integral}}] For   $x \in \R^N\setminus \{0\},$ using the polar coordinates, we  write $x=(r,\omega) \in (0,\infty) \times \mathbb{S}^{N-1}$. Thus for $u \in \mathcal{C}_c^{\infty}(\R^N)$, 
\begin{align}
u(r,\omega)&= - \int_r^{\infty} \frac{\partial u}{\partial t} (t,\omega)\dt=-r \frac{\partial u}{\partial t}(r,\omega) + \int_r^{\infty}t \frac{\partial^2 u}{\partial t^2}(t,\omega) \dt   \notag\\ 
   &= \int_r^{\infty} (t-r)\frac{\partial^2 u}{\partial t^2}(t,\omega)\dt.        \label{polarineq}  
\end{align}
Hence
\begin{align*}
|u(r,\omega)| &\leq  \int_r^{\infty}t \left|\frac{\partial^2 u}{\partial t^2}(t,\omega)\right|\dt     = \int_r^{\infty} t\ t^{\frac{1-N}{2}} t^{\frac{N-1}{2}} \left|\frac{\partial^2 u}{\partial t^2}(t,\omega)\right|\dt.
\end{align*}
 Now by H\"older inequality, we get 
 \begin{align}
 |u(r,\omega)|^2 &\leq \left( \int_r^{\infty} t^2 t^{1-N}\dt \right) \left( \int_r^{\infty} t^{N-1}\left|\frac{\partial^2 u}{\partial t^2}(t,\omega)\right|^2\dt \right) \notag\\
 &=\frac{1}{N-4} r^{4-N}  \int_r^{\infty} t^{N-1}\left|\frac{\partial^2 u}{\partial t^2}(t,\omega)\right|^2\dt.   \label{polarineq1}
 \end{align}
 Multiply both  sides of \eqref{polarineq1} by $w(r)$ and  integrate over $\mathbb{S}^{N-1}$ to obtain
 \begin{align}
 \int_{\mathbb{S}^{N-1}} |u(r,\omega)|^2 w(r) \diff S_{\omega} &\leq \frac{1}{N-4} r^{4-N}  w(r) \int_0^{\infty} \int_{\mathbb{S}^{N-1}} t^{N-1}\left|\frac{\partial^2 u}{\partial t^2}(t,\omega)\right|^2 \diff S_{\omega} \dt \no\\& \leq \frac{1}{N-4} r^{4-N}  w(r)\left( \int_{\R^N} |\Delta u|^2 \dx\right),   \label{exterior higher2}
 \end{align}
 where the last inequality follows from  Lemma~\ref{Polarinequality}.
Finally, multiplying  both the sides of (\ref{exterior higher2}) by $r^{N-1}$ and integrating over $(0,\infty)$ with respect to $r$ yields:
 \begin{align*}
 \int_{\R^N} w(|x|)u^2 \dx \leq \frac{N \omega_N}{N-4} \left( \int_0^{\infty} w(r)r^3 \dr\right) \int_{\R^N} |\Delta  u|^2 \dx.
 \end{align*} 
In particular, as $g(x)\le w(|x|)$    we have 
  \begin{align*}
 \int_{\Om} g(x)u^2 \dx \leq \frac{N \omega_N}{N-4} \left( \int_0^{\infty} w(r)r^3 \dr\right) \int_{\Om} |\Delta  u|^2 \dx, \; \forall u\in \cc(\Om).
 \end{align*}
 Now by  density of  $ \cc(\Om)$, the  above inequality holds for all $u\in \d2$ and hence $g$ is admissible.
\end{proof}

 Observe that, we have two different set of conditions for the admissibility from Theorem \ref{higherdimesionthm} and Theorem \ref{Thm:Integral}. Next  examples  show that these two conditions are independent, i.e., one does not imply the other. 

      \begin{example}\rm
       Let $\Om=\mathbb{R}^N $ with $N\ge 5$ and let $\be\in (\frac{4}{N},1)$.  Consider
      \begin{equation*}
      g_1(x)= \left\{
          \begin{array}{ll}
          (|x|-1)^{-\be}, & 1<|x|\leq 2, \\
                                                         0 , & {\mbox{otherwise}}.
                                                          \end{array}\right .
      \end{equation*}
              We can compute the distribution function $\al_{g_1}$ and the one dimensional decreasing rearrangement 
 $g_1^* $  as below: 
        $$ \alpha_{g_1} (s)= \left\{
          \begin{array}{ll}
                    \omega_N 2^N -\om_N  &, \ 0 \leq s <1, \\ 
           \omega_N\left( s^{-\frac{1}{\be}}+1 \right)^N -\om_N &, \ s \geq 1,  
           \end{array} 
        \right.$$$$           
                  g_1^*(t)= \left\{ 
                      \begin{array}{ll}
                    0  & ,\ t > \omega_N(2^N-1),\\
                      \bigg(\bigg(\frac{t}{\omega_N}+1\bigg)^{\frac{1}{N}} -1\bigg)^{-\be} &,\  t \leq \omega_N(2^N-1).
                      \end{array}
                      \right .$$ 
          Hence, for $t \leq  \omega_N(2^N-1),$ 
          \begin{align*}
          t^\frac{4}{N} g_1^*(t) = & \ t^\frac{4}{N}\bigg(\bigg(\frac{t}{\omega_N}+1\bigg)^{\frac{1}{N}} -1\bigg)^{-\be}  \geq   t^\frac{4}{N}\bigg(\bigg(\frac{t}{\omega_N}+1\bigg) -1\bigg)^{-\be} = \ t^\frac{4}{N}\bigg(\frac{t}{\omega_N}\bigg)^{-\be}.
          \end{align*}
          Since $\be>\frac{4}N$, $ \sup_{t\in (0,\infty)} t^\frac{4}{N} g_1^*(t)= \infty $ and hence $g\notin L^{\frac{N}{4},\infty}(\R^N)$.\\
          Let $w(r)=  (r-1)^{-\be}\chi_{(1,2)}(r).$ Clearly $g_1(x)\le w(|x|), \forall x\in \R^N$ and since $\be<1,$
          $$\int_0^\infty w(r) r^3 \dr = \int_1^2 (r-1)^{-\be}r^3\dr\le 8\int_0^1 s^{-\be}\ds<\infty.$$
          Thus $g_1$ is admissible by Theorem \ref{Thm:Integral}.                              
           \end{example}
             \begin{example}
      {\rm   Let  $g_2(x)= \frac{1}{|x|^4},\, x\in \R^N$ with $N\ge 5.$ By Example \ref{Hardy-Rellich},  $g_2\in L^{\frac{N}{4},\infty}(\R^N)$  and hence  admissible by Theorem \ref{higherdimesionthm}. Let $w$ be a function on $(0,\infty)$ such that
      $g(x)\le w(|x|).$ Then 
      $$\int_0^\infty w(r)r^3\ge \int_0^\infty r^{-4} \times r^3 \dr= \infty.$$
      Thus $g_2$ does not satisfy the assumptions of Theorem \ref{Thm:Integral}. 
  }
                                            \end{example}
                       \begin{remark}  \rm{The above examples shows that the the sufficient conditions given by Theorem~\ref{higherdimesionthm} and Theorem~\ref{Thm:Integral} are independent. The question whether these conditions exhaust all the admissible weights or not  is open.}
                         \end{remark}
                       \begin{remark} {\rm There are admissible weights whose Schwarz symmetrization are not admissible.  For example, the Schwarz symmetrization  $g_1^\star$ of $g_1$  does  not belong to $L^{\frac{N}{4},\infty}(\Om^\star)$ and hence by part of (ii) of Theorem~\ref{higherdimesionthm},  $g_1^\star$ can not be an admissible weight.}
                       \end{remark}

  \subsection{The critical dimension ($N=4$) and the lower dimensions ($N=1,2 $  and $ 3$)}
Now we consider the cases $\Om$ is a bounded open set or an exterior domain. In either cases, the space $\d2$ is a well defined function space. 
First we give a proof of Theorem \ref{dim4thm}. 
   \begin{proof}[{\bf Proof of Theorem~\ref{dim4thm}}]
 (i) {\bf A sufficient condition}.  The proof follows in the same line as in the proof of Theorem \ref{higherdimesionthm}.
 Let $u\in  \cc(\Om)$. Then by the Hardy-Littlewood inequality \eqref{HardyLittle} we have 
\begin{equation}\label{Hardy-Littlewood1}
\int_\Om g(x) u(x)^2\dx \le \int_0^{|\Omega|} g^*(s) u^*(s)^2\ds.
\end{equation}  
Further, using \eqref{Lem:Cianchi} we have
\begin{align}\label{suffdim4}
 \int_0^{|\Omega|} g^*(s) u^*(s)^2 \ds &\leq 2\int_0^{|\Omega|} g^*(s)s^{-1} \left(\int_0^s|\Delta u|^*(t)\dt \right)^2\ds   \notag \\ 
       \quad &+ 2\int_0^{|\Omega|} g^*(s) \left(\int_s^{\infty}|\Delta u|^*(t)t^{\frac{1}{2}}\dt \right)^2\ds \no \\
       & \le C \norm{g}_{\l2(\Om)} \int_0^{|\Om|}\left(|\Delta u|^*(t)\right)^2 \dt,
 \end{align}
 where the last inequality follows from Lemma~\ref{Lem:higher}, as  $g\in  \l2(\Omega).$ From 
 \eqref{Hardy-Littlewood1} and \eqref{suffdim4}, we get 
\begin{align*}
\int_\Om g(x) u(x)^2\dx \leq \int_0^{|\Omega|} g^*(s) u^*(s)^2\ds \leq  C \norm{g}_{\l2(\Om)} \int_\Om |\De u|^2 \dx,
 \end{align*}
Thus by  density, the above inequality holds for all $u$ in $\d2$ and hence  $g$ is admissible.
 
\noi (ii) {\bf A necessary condition}. 
 Let $R\in (0,\infty)$ and let $\Om=B(0;R)\subset \R^4$. Let $g$ be a nonnegative, radial and radially decreasing  admissible function on $\Om$. 
 To  show  $g \in \l2,$ for each $r\in (0,R)$, we consider the following test function:
 \begin{align*}
           u_r(x)= \left\{
          \begin{array}{ll}
           \frac{1}{e^2} \bigg (\log (\frac{R}{r}) \bigg )^2, & |x|\leq r \\
                      \bigg ( \log (\frac{R}{|x|}) \bigg )^2  \Phi_r (x), & r<|x|<R
                                                          \end{array}\right.
\end{align*}
 where $\Phi_r (x)= \exp{\bigg ( -\frac{2 \log (\frac{R}{|x|})}{\log (\frac{R}{r})}\bigg)}$. In our computations we use the notation $D_i \equiv \frac{\partial}{\partial x_i}$ and $D_{ii} \equiv \frac{\partial^2}{\partial x_i^2}.$   For $r\le |x|\le R,$ noting that 
 $D_i \Phi_r(x)=\frac{2 x_i}{|x|^2\log (\frac{R}{r})} \Phi_r(x) $ and $D_i \log(\frac{R}{|x|} 
)=-\frac{x_i}{|x|^2},$  we compute the derivatives of $u_r$ as below:
 \begin{align*}
D_iu_r(x)=  2\Phi_r (x)\frac{x_i}{|x|^2}\log \left(\frac{R}{|x|}\right)\left[\frac{\log(\frac{R}{|x|})}{\log (\frac{R}{r})} - 1\right]. 
 \end{align*}
 Furthermore, 
 \begin{align*}
D_{ii}^2 u_r(x)= &\Phi_r(x)\left\{ \frac{4 x_i^2}{|x|^4}\frac{\log(\frac{R}{|x|})}{\log (\frac{R}{r})}+ 2\left(-\frac{2 x_i^2}{|x|^4}+\frac{1}{|x|^2}\right) \log\left(\frac{R}{|x|}\right)-\frac{2x_i^2}{|x|^4}\right\} \left[\frac{\log(\frac{R}{|x|})}{\log (\frac{R}{r})} - 1\right]\\
&-2\Phi_r (x)\frac{x_i^2}{|x|^4}\frac{\log(\frac{R}{|x|})}{\log (\frac{R}{r})}.
\end{align*}
Thus for $r\le |x|\le R$,
\begin{align*}
 \De u_r=&\Phi_r(x)\left\{\frac{4\log(\frac{R}{|x|})}{ |x|^2\log (\frac{R}{r})}+ \frac{4}{|x|^2} \log\left(\frac{R}{|x|}\right)-\frac{2}{|x|^2}\right\} \left[\frac{\log(\frac{R}{|x|})}{\log (\frac{R}{r})} - 1\right]-2\Phi_r (x)\frac{1}{|x|^2}\frac{\log(\frac{R}{|x|})}{\log (\frac{R}{r})}.
\end{align*}
Observe that $\Phi_r (x)\le 1$, $\log(\frac{R}{|x|})\le \log (\frac{R}{r}) $ and  $1\le \log (\frac{R}{r})$  for $r\le \frac{R}{e}$. Hence   
 \begin{align*}
|\De u_r(x)| \leq \frac{16}{|x|^2}\log\left(\frac{R}{|x|}\right)+\frac{4}{|x|^2}+ \frac2{|x|^2}\log\left(\frac{R}{|x|}\right) \le\frac{18}{|x|^2}\log \left(\frac{R}{|x|}\right)+\frac{4}{|x|^2}.
\end{align*}
 Thus for $r<\frac{R}{e},$ we have
 \begin{align}\label{critic:eqn1}
  \int_{\Om}|\De u_r(x)|^2 &\le  C_1 \int_{\Om \setminus B(0,r)}\left[\frac{1}{|x|^4}\log \left(\frac{R}{|x|}\right)^2 + \frac1{|x|^4}\right]\dx \no \\ & \leq C_1 \left\{\left[\log\left(\frac{R}{r}\right)\right]^3+\log\left(\frac{R}{r}\right)\right\} \leq C_1 \left[\log\left(\frac{R}{r}\right)\right]^3 ,
 \end{align}
 where $C_1$ is a positive constant independent of $r.$ Notice that $u_r$ is a $C^1$ function such that  $u_r$ and $\nabla u_r$ vanish when $|x|=R,$ hence $u_r\in H_0^2(\Om).$ 
 Further, as $g$ is radial, radially decreasing, we easily obtain the following estimate:
   \begin{align}\label{critic:eqn2}
    \int_\Om g(x) u_r^2(x) \dx \ge \int_{B(0,r)} g(x) u_r(x)^2\dx = \left[\frac1{e^2}\log\left(\frac{R}{r}\right)\right]^4 \int_0^{\om_4 r^4} g^*(s) ds
   \end{align}
  Now the admissibility of $g$ together with \eqref{critic:eqn1} and \eqref{critic:eqn2} yields  
  $$\displaystyle  \log\left(\frac{R}{r}\right) \int_0^{\om_4 r^4} g^*(s) ds \leq C, \quad \forall\, r\in (0,\frac{R}{e}).$$ 
  By taking $t=\om_4 r^4$, we get 
  $$\displaystyle  \frac14\log\left(\frac{|\Om|}{t}\right) \int_0^{t} g^*(s) ds \leq C, \quad \forall\, t\in \left(0,\frac{|\Om|}{e^4}\right).$$
  Since $t g^{**}(t)\log(\frac{|\Om|}{t})$ is bounded on $ \frac{|\Om|}{e^4} \leq t \leq |\Om|,$ from the above inequality we conclude that 
 $$ \sup_{t\in (0,|\Om|)} t g^{**}(t)\log(\frac{|\Om|}{t})< \infty.$$
  Hence $g\in \l2(\Omega).$
   \end{proof}

As a corollary of the sufficiency part of our previous theorem, we give a simple alternate proof for the embedding of $H^2_0(\Om)$ into the Lorentz-Zygmund space $L^{\infty, \;2}(\log L)^{-1}(\Omega)$ obtained independently by Brezis and Wainger \cite{Brezis_Wainger} and Hansson \cite{Hansson}.
   
   \begin{corollary}\label{CriticalEmbd}
	Let $\Omega \subset \mathbb{R}^4$ is an open bounded set. Then we have the following embedding:$$H^2_0(\Om) \hookrightarrow L^{\infty, \;2}(\log L)^{-1}(\Omega).$$
\end{corollary}

\begin{proof} First, assume that $\Om$  is  a ball of radious $R$ with center at the origin. Let $X=\l2(\Om).$ 
For each $g \in X,$ \eqref{suffdim4} gives, $$ \int_0^{|\Om|} g^*(t)(u^*(t))^2 \dt  \leq C \norm{g}_X \int_{\Om} |\De u|^2 \dx,\ \forall u \in H_0^2(\Om).$$
	Let $g_1(x)= \left[|x|^2 \log((\frac{R}{|x|})^4e)\right]^{-2}, x \in \Om.$ We calculate, \tr{$g_1^*$} $g_1^{**}(t)= \displaystyle\frac{\om_4}{t(\log{\frac{e|\Om|}{t}})} , t\in (0,|\Om|).$ Therefore, $g_1\in X $ and $ \norm{g_1}_X=\om_4$. Thus by the above inequality we have
	$$  \int_0^{|\Om|} \displaystyle\frac{(u^*(t))^2}{t \left[\log \left(\frac{e|\Om|}{t}\right)\right]^2} {\dt} \leq C_1 \int_{\Om} |\De u|^2 \dx ,\ \forall u \in H_0^2(\Om).$$
	The left hand side of the above inequality is equivalent to ${\norm{u}}^2_{L^{\infty ,\; 2} (\log L)^{-1}(\Omega)}$(Proposition \ref{equivnorm}). Therefore,
$$ \norm{u}_{L^{\infty ,\; 2} (\log L)^{-1}(\Omega)}^2 \le  C_2 \int_{\Om} |\De u|^2 \dx , \forall u\in H_0^2(\Om).$$
		Now for a general bounded set $\Omega,$ there exists $ R >0$ such that $\Omega \subset B(0,R).$ In this case,  we obtain the required embedding by considering the above inequality for the zero extension to $ B(0,R).$
	
\end{proof}

     \begin{remark}\rm \label{rem1dim4}
     For a bounded open set, we have the following continuous inclusions: $$L^{\infty ,\; 2} (\log L)^{-1}(\Omega) \hookrightarrow L_{e^{t^2}-1}(\Om) \hookrightarrow L^p(\Om),1\le p <\infty.$$ Thus  the above embedding gives the classical  Sobolev embedding and Adams' embedding: 
     $$ H_0^2(\Om) \hookrightarrow L^p(\Om), 1\le p <\infty;\quad H_0^2(\Om) \hookrightarrow L_{e^{t^2}-1}(\Om).$$
        \end{remark}

Next we give a proof of Theorem~\ref{Thm:exterior} for the cases $N=2,3,4.$

\begin{proof}[Proof of Theorem~\ref{Thm:exterior} for the cases $N=2,3,4$. ]
	
As before, for $x \in B_R \setminus B_1$, we write $x=(r,\omega) \in (1,R) \times \mathbb{S}^{N-1} $. For $u\in \cc(\Om),$ we use the fundamental theorem of calculus to get 
	\begin{equation*}
		u(r,\omega)=  \int_1^r \frac{\partial u}{\partial t} (t,\omega)\dt.
	\end{equation*}
	As in the proof of Theorem~\ref{Thm:Integral}, we deduce
	\begin{equation*} 
		u(r,\omega) = \int_1^r (r-t) \frac{\partial^2 u}{\partial t^2}(t,\omega)\dt
		= \int_1^r (r-t) t^{-\frac{N-1}{2}} t^{\frac{N-1}{2}}\frac{\partial^2 u}{\partial t^2}(t,\omega)\dt.
	\end{equation*} 
	Now  H\"older inequality yields
	\begin{align*}  
		|u(r,\omega)|^2 &\leq r^2 \bigg(\int_1^r    t^{-(N-1)} \dt\bigg) \bigg( \int_1^r t^{N-1}\bigg|\frac{\partial^2 u}{\partial t^2}(t,\omega)\bigg| ^2dt\bigg) .
	\end{align*}
	 Multiply the above inequality by $r^{N-1}w(r)$ and integrate over $\mathbb{S}^{N-1} \times (1,R)$ and use   Lemma~\ref{Polarinequality} to obtain 
	\begin{align} 
		\int_{B_R\setminus B_1} w(|x|)u^2 \dx &\leq \bigg(\int_1^{R}  \bigg[\int_1^r  t^{1-N} \dt\bigg]r^{N+1} w(r) \dr\bigg) \int_{\mathbb{S}^{N-1}} \int_1^R t^{N-1}\bigg|\frac{\partial^2 u}{\partial t^2}(t,\omega)\bigg| ^2dt \diff S_{\omega}\no\\
		& \leq I \times \bigg(\int_{B_R\setminus B_1}|\Delta u|^2 \dx\bigg), \label{exteriorlower3}
	\end{align}
	where $ I=\bigg(\int_1^{R} r^{N+1} \bigg[\int_1^r  t^{1-N} \dt\bigg] w(r) \dr\bigg)$. Notice that 
	\begin{align}\label{exteriorlower4}
		I \leq \begin{cases}
			\int_{1}^{\infty}  r^{N+1} w(r) \dr,   \ \ \ N=3,4; R=\infty\\ 
			\int_1^{\infty} r^3 (\log r) w(r )\dr, \ \ \ N=2; R =\infty\\
			\int_1^R w(r)\dr,\ \ \  2 \leq N \leq 4; R<\infty.
		\end{cases}
	\end{align}
	 Therefore, the assumptions on $g$ together with \eqref{exteriorlower3} and \eqref{exteriorlower4} gives the admissibility of $g$. 	 
\end{proof}

\begin{remark}\label{embeddings}{\rm 
Let $f\in L^1(1,\infty)$ and $f\ge 0.$ Then by  Theorem \ref{Thm:exterior}, we have the following embeddings:
 $$\D^{2,2}_0(R^N\setminus \bar{B_1}) \hookrightarrow \left \{ \begin{array}{ll}
 L^2(\Om,\frac{f(|x|)}{|x|^{N+1}}), & N= 3,4; \\
 L^1(\Om,\frac{f(|x|)}{|x|^3\log(|x|)}),  & N=2;
 \end{array} \right. .$$
 For example, one can take $f(r)=\frac{1}{r^2}.$ Thus $\D^{2,2}_0(R^N\setminus \bar{B_1})$ is always a well defined function space. }
 \end{remark}

\begin{remark}
\rm For an annular region $\Om \subset \R^4,$  let $g$ be a nonnegative radial  function such that   $g\in L^1(\Om)$  and  $g\notin \l2(\Om)$.   Then   $g$  is admissible by the above theorem, however $g^\star$ is not admissible on $\Om^\star$ by  part (ii) of Theorem \ref{dim4thm}. An example of such  weight is 
 $ g(x)= \displaystyle\frac{1}{(|x|^4 -1)( \log\frac{16e}{(|x|^4 -1)})^{\frac{3}{2}}}$ on $B(0;2) \setminus B(0;1)$.
\end{remark}

\begin{theorem}
 Let $\Om$ be a bounded open set  in $\R^N$ with $N=1,2$ or 3. Then $g\in L^1(\Om)$ with $g\ge 0$ is admissible.
\end{theorem}
\begin{proof}
For $N=1,2,3,$  $H^2_0(\Om)$ is continuously embedded into $L^{\infty}(\Om)$.
Hence, for $g\in  L^1(\Om),$ 
$$ \int_{\Om} gu^2 \dx  \leq \norm {g}_{1} {\norm u}_{\infty}^2\le C \norm {g}_{1} \int_{\Om} |\Delta u|^2 \dx, \forall\, u\in H^2_0(\Om),$$
where $C$ is the embedding constant. 
\end{proof}

 \begin{remark}
 \rm Let $\Om$ and $X$ be as in Lemma \ref{Lem:higher} and let $$\F_X= \overline{\cc(\Om)}\subset X.$$  Then for $g\in \F_X$, the best constant in inequality \eqref{GHR} is attained. This will follow as the map $G:\d2 \ra \R$ defined by $G(u)=\int_\Om g u^2\dx$ is compact. The compactness of $G$ can be prove using a similar set of arguments as in the proofs of Lemma 15 of \cite{anoop} and Lemma 5.1 \cite{ALM}.
 \end{remark}

  \begin{appendices}
\section{}
First we see that  as a vector space  $\L2(\Om)$ and  $\l2(\Om)$ are same.  
\begin{proposition}\label{equiv}
     Let  $\Omega \subset \mathbb{R}^N$ be a bounded open set. Then $\L2(\Om) =  \l2(\Om)$.
  \end{proposition}
  \begin{proof}
  First we show that $\ \norm{f}_{\l2}  \le  |f|_{(1,\infty,2)}.$  For  $f \in \Me(\Om)$ and   $t\in (0,|\Omega|),$  we have
  \begin{align*}
 t f^{**}(t)  &= \int_0^t f^*(s) \ds =
   \int_0^t f^*(s) s \left [\log \left( \frac{|\Omega|}{s}\right) \right ] ^2 \frac{1}{s \left[ \log \left( \frac{|\Omega|}{s}\right) \right ]^2} \ds \\
  &\leq \sup_{0<s\leq t}f^{*}(s) \,s\, \left [\log \left( \frac{|\Omega|}{s}\right) \right ]^2 \int_0^t  \frac{1}{s \left [\log \left( \frac{|\Omega|}{s}\right)\right]^2} \ds\\
  &\le  |f|_{(1,\infty,2)}\frac{1}{  \log \left( \frac{|\Omega|}{t}\right)}.
  \end{align*}
     This yields   $
  \norm{f}_{\l2}  \le  |f|_{(1,\infty,2)}
  $
  and hence  $$\L2(\Om) \subseteq  \l2(\Om).$$
  If  the above inclusion is strict, then $\exists f \in \l2(\Om)\setminus  \L2(\Om),$ i.e.,    $$\sup_{0<t< |\Om|}  f^{**}(t) \,t\left [\log \left( \frac{|\Omega|}{t}\right) \right ]<\infty;\quad \sup_{0<t< |\Om|}f^{*}(t) \,t\,\left [ \log \left( \frac{e |\Omega|}{t}\right) \right ]^2 = \infty.$$ Now  consider the function $$g(t)=f^{*}(t) \,t\,\left [ \log \left( \frac{e |\Omega|}{t}\right) \right ]^2,  0<t<|\Om|.$$
 Claim: $\displaystyle \lim_{t\ra 0} g(t) = \infty.$ \\
If  the claim is not true, then  $\exists t_0>0$ such that $\sup_{t\ge t_0} g(t)  = \infty.$  Since $\,t\,\left [ \log \left( \frac{e |\Omega|}{t}\right) \right ]^2$ is bounded, we must have $f^*(t)=\infty$ for $t\le t_0$. A contradiction as $f\in \l2(\Om)$ hence claim must be true.\\
Now by the claim,  there exists a decreasing sequence $(t_n)$ in $(0,|\Om|)$ such that $(t_n)$  converging to $0$ and  $g(t)> n,$ for $t\in (0, t_n).$ 
  Consequently, 
  \begin{align*}
  t_n f^{**}(t_n) &=  \int_0^{t_n} \frac{g(t)}{t\,\left [ \log \left( \frac{e |\Omega|}{t}\right) \right ]^2} dt
     \geq n \int_0^{t_n} \frac{1}{t\,\left [ \log \left( \frac{e |\Omega|}{t}\right) \right ]^2} dt \ge   \frac{n}{\log \left( \frac{e |\Omega|}{t_n}\right)}.
  \end{align*}
  Therefore,  $$\lim_{n \ra \infty} t_n f^{**}(t_n) \log \left( \frac{|\Omega|}{t_n}\right) \ge \lim_{n \ra \infty} n \frac{\log \left( \frac{|\Omega|}{t_n}\right) }{\log \left( \frac{e |\Omega|}{t_n}\right)}=\infty.$$  A contradiction as $f \in \l2(\Om).$ Hence we must have $\L2(\Om) =  \l2(\Om)$. 
 
  \end{proof}
  
  \begin{remark}{\rm
   The quasinorm $|f|_{(1,\infty,2)}$ and the norm $\norm{f}_{\l2}$ defines the same vector space, however, they are not equivalent.  To see this,  let  $\Om=B(0;R)  \subset \R^N$ and for each $n\in \N,$ consider the function $\{f_n \}$ on $\Om $ defined as
  $$f_n(x)= \displaystyle\frac{1}{|x|^N [\log((\frac{R}{|x|})^N e)]^{n+2}}.$$ 
Thus we have $f_n^*(t)= \displaystyle\frac{\om_N}{t [\log(\frac{e |\Om|}{t})]^{n+2}}$ and $f_n^{**}(t)= \displaystyle\frac{\om_N}{(n+1) t [\log(\frac{e |\Om|}{t})]^{n+1}}.$
 Therefore, \begin{align*}
 |f_n|_{(1,\infty,2)}&=\sup_{0<t<|\Om|} t\left[\log\left(\frac{e|\Om|}{t}\right)\right]^2f_n^*(t) = \om_N\sup_{0<t<|\Om|} \frac{1 }{ \left[\log\left(\frac{e|\Om|}{t}\right)\right]^{n}},\\
  \norm{f_n}_{\l2}&=\sup_{0<t<|\Om|} t\left[\log\left(\frac{|\Om|}{t}\right)\right]f_n^{**}(t)\le \frac{\om_N}{n+1}\sup_{0<t<|\Om|} \frac{1 }{ \left[\log\left(\frac{e|\Om|}{t}\right)\right]^{n}}.
  \end{align*}
 Hence $(n+1)\norm{f_n}_{\l2}\le  |f_n|_{(1,\infty,2)}.$}
  \end{remark}
The next proposition provides  the equivalence of the quasinorm $ |u|_{(\infty,1,-2)}$ and the norm $ \norm{u}_{(\infty,1,-2)}.$ We adapt the proof
  of   Theorem 6.4 of \cite{ Bennet} to our case.
  \begin{proposition}\label{equivnorm}
 Let $\Om$ be a bounded subset of $\R^N$ and $u: \Om \ra \R$ be a measurable function. Then there exist a constant $C> 0$ such that
 $$ \int_0^{|\Om|} \left(\frac{ u^{**}(t)}{\log(\frac{e|\Om|}{t})} \right)^2 \frac{dt}{t} \leq C \int_0^{|\Om|} \left(\frac{ u^{*}(t)}{\log(\frac{e|\Om|}{t})}\right)^2 \frac{dt}{t}.$$
\end{proposition}
\begin{proof}
 Choose $ 0 < \delta <1$ and write $u^*(s)=[s^{\delta}u^*(s)][s^{1-\delta}] s^{-1}$. Using Holder's inequality we obtain,
 \begin{eqnarray}
 \left(\int_0^t u^*(s) ds \right)^2 \leq C_1 t^{2- 2 \delta} \left(\int_0^t [s^{\delta}u^*(s)]^2 \frac{ds}{s} \right).
  \end{eqnarray}
  Multiplying by $\displaystyle \left(\frac{1}{t^3 (\log(\frac{e|\Om|}{t}))^2 } \right)$ and integrating over $(0,|\Om|)$ we get
 \begin{eqnarray}\label{embedprop1} 
   \int_0^{|\Om|} \left(\frac{ u^{**}(t)}{\log(\frac{e|\Om|}{t})} \right)^2 \frac{dt}{t} & \leq & C_1 \int_0^{|\Om|} \frac{1}{t^{2 \delta} (\log(\frac{e|\Om|}{t}))^2} \left( \int_0^t [s^{\delta}u^*(s)]^2 \frac{ds}{s}  \right) \frac{dt}{t} \nonumber \\
                                                                                         & \leq & C_1 \int_0^{|\Om|} [s^{\delta}u^*(s)]^2  \left( \int_s^{|\Om|} \frac{1}{t^{2 \delta} (\log(\frac{e|\Om|}{t}))^2} \frac{dt}{t} \right)  \frac{ds}{s} \nonumber \\
                                                                                         & \leq & C_1 \int_0^{|\Om|} \frac{[s^{\delta}u^*(s)]^2}{s^{\delta} (\log(\frac{e|\Om|}{s}))^2 } \left( \int_s^{|\Om|} \frac{dt}{t^{1+ \delta}} \right)  \frac{ds}{s} .
                                                                                         \end{eqnarray} 
 The last two inequalities of $\eqref{embedprop1}$ follows from Fubini's theorem and monotonic decreasing property of $ \displaystyle\frac{1}{t^{\delta} (\log(\frac{e|\Om|}{t}))^2 }$ respectively. Further, we estimate the right hand side of
 $\eqref{embedprop1}$  as below,
 \begin{eqnarray}\label{embedprop2} 
   \int_0^{|\Om|} \frac{[s^{\delta}u^*(s)]^2}{s^{\delta} (\log(\frac{e|\Om|}{s}))^2 } \left( \int_s^{|\Om|} \frac{dt}{t^{1+ \delta}} \right)  \frac{ds}{s} \leq C_2 \int_0^{|\Om|} \frac{[s^{\delta}u^*(s)]^2}{s^{\delta} (\log(\frac{e|\Om|}{s}))^2 } \left( \frac{1}{s^{\delta}} \right)  \frac{ds}{s}.
  \end{eqnarray}
  Hence by combining $\eqref{embedprop1}$ and $\eqref{embedprop2}$ we have the following inequality   as required
  $$ \int_0^{|\Om|} \left(\frac{ u^{**}(t)}{\log(\frac{e|\Om|}{t})} \right)^2 \frac{dt}{t} \leq C \int_0^{|\Om|} \left(\frac{ u^{*}(t)}{\log(\frac{e|\Om|}{t})}\right)^2 \frac{dt}{t}.$$
\end{proof}
  \end{appendices}

\bibliography{ref1}
\bibliographystyle{abbrv}
{\bf  T. V.  Anoop } \\  Department of Mathematics,\\   Indian Institute of Technology Madras, \\ Chennai, 600036, India. \\ 
{\it Email}:{ anoop@iitm.ac.in}
		
	\noi {\bf	 Ujjal Das }\\  The Institute of Mathematical Sciences, HBNI \\ Chennai, 600036, India. \\
	 {\it Email}: ujjaldas@imsc.res.in, ujjal.rupam.das@gmail.com
		
	 \noi {\bf  Abhishek Sarkar}\\ NTIS, University of West Bohemia\\
		Technick\'{a} 8, 306 14 Plze\v{n}, Czech Republic.\\
		{\it Email}:{ sarkara@ntis.zcu.cz}
		
\end{document}